\documentclass[a4paper,fleqn,12pt]{article}

\usepackage{epsfig}
\usepackage[shortlabels]{enumitem}
\usepackage{amsmath}
\usepackage{amssymb}
\usepackage{amsthm}

\newcommand{\commentaar}[1]{{}}

\newtheorem{theorem}{Theorem}[section]
\newtheorem{lemma}[theorem]{Lemma}
\newtheorem{proposition}[theorem]{Proposition}

\newtheorem{definition}[theorem]{Definition}
\newtheorem{Remark}[theorem]{Remark}

\newenvironment{remark}{\begin{Remark} \begin{rm}}{\end{rm} \end{Remark}}

\newcommand{\basis}[1]{\{ #1 \}}
\newcommand{\breukje}[2]{\mbox{$\tfrac{#1}{#2}$}}
\newcommand{\inprod}[2]{\langle #1 \mid #2 \rangle}
\newcommand{\partieel}[2]{\frac{\partial #1}{\partial #2}}

\newcommand{\poisson}[2]{\{#1,#2\}}

\newcommand{\Diff}{\text{Diff}}
\newcommand{\GGG}{G}
\newcommand{\N}{{\mathbb N}}
\newcommand{\Q}{{\mathbb Q}}
\newcommand{\Ri}{\R[[I]]}
\newcommand{\Ris}{\R[[I]]_0/{\sim}}
\newcommand{\Rhk}{\R[[H,K]]_0}
\newcommand{\Rtdv}{\R[[I_2,I_3,I_4]]_0} 
\newcommand{\Retdv}{\R[[I_1,I_2,I_3,I_4]]_0}
\newcommand{\R}{{\mathbb R}}

\newcommand{\X}[1]{X_{\textstyle \! #1}}
\newcommand{\Z}{{\mathbb Z}}
\newcommand{\becomes}{\mathrel{\mathop:}=}
\newcommand{\cHhk}[1]{\mathcal{H}_{#1}(H,K)}
\newcommand{\cHtdv}[1]{\mathcal{H}_{#1}(I_2,I_3,I_4)}
\newcommand{\cF}{\mathcal{F}}
\newcommand{\cG}{\mathcal{G}}
\newcommand{\cJ}{\mathcal{J}}
\newcommand{\cM}{\mathcal{M}}
\newcommand{\cR}{\mathcal{R}}
\newcommand{\cX}{\mathcal{X}}
\newcommand{\cinf}{C^{\infty}}
\newcommand{\degree}{\text{degree}}
\newcommand{\e}{\text{e}}
\newcommand{\halfje}{\breukje{1}{2}}
\newcommand{\hot}{\text{h.o.t.}}
\newcommand{\hpm}{\hphantom{-}}

\newcommand{\mitdE}{\mbox{\large $ \;\; | \;\; $}}
\newcommand{\sL}{\mathsf{L}}
\newcommand{\spansel}[1]{\langle #1 \rangle}
\newcommand{\vvx}{\tfrac{\partial}{\partial I_2}}
\newcommand{\vvy}{\tfrac{\partial}{\partial I_3}}
\newcommand{\vvz}{\tfrac{\partial}{\partial I_4}}

\commentaar{alle commentaren langs om te zien of er wat mee gedaan moet worden}

\newcommand{\version}{7 April 2017}
\newcommand{\shorttitle}{The $1{:}1$ resonance}
\newcommand{\shortauthor}{Han{\ss}mann, Hoveijn}

\title{The $1$:$1$ resonance in Hamiltonian systems}
\author
{
{\protect\normalsize Heinz Han{\ss}mann} \\
{\protect\footnotesize\protect\it
Mathematisch Instituut, Universiteit Utrecht, Postbus 80.010,
3508~TA~Utrecht, The Netherlands}
\\[2mm]
{\protect\normalsize Igor Hoveijn} \\
{\protect\footnotesize\protect\it 
Langewoldlaan 2, 9727~DD~Groningen, The Netherlands}
}
\date{\version}

\begin{document}

\maketitle

\tableofcontents

\vspace{10mm}
\noindent

\hfill \textit{The case $k=l=1$} ($1{:}1$ resonance)
       \textit{turns out to be surprisingly complicated.}

\hfill \textit{J.J. Duistermaat in \cite{duistermaat1984}}


\begin{abstract}
\noindent
Two-degree-of-freedom Hamiltonian systems with an elliptic
equilibrium at the origin are characterised by the frequencies
of the linearisation.
Considering the frequencies as parameters, the system undergoes
a bifurcation when the frequencies pass through a resonance.
These bifurcations are well understood for most resonances $k{:}l$,
but not the semisimple cases $1{:}1$ and $1{:}{-}1$.
A two-degree-of-freedom Hamiltonian system can be approximated to
any order by an integrable approximation.
The reason is that the normal form of a Hamiltonian system has an
additional integral due to the normal form symmetry.
The latter is intimately related to the ratio of the frequencies.
Thus we study $S^1$--symmetric systems.
The question we wish to address is about the co-dimension of such a
system in $1{:}1$~resonance with respect to left-right-equivalence,
where the right action is $S^1$--equivariant.
The result is a co-dimension five unfolding of the central
singularity.
Two of the unfolding parameters are moduli and the remaining
non-modal parameters are the ones found in the linear unfolding
of this system.
\end{abstract}

\section{Introduction}\label{sec:intro}
One of the few available methods to study the dynamics of Hamiltonian
systems is to concentrate on the equilibria.
The motion itself being trivial by definition, one considers the local
dynamics and linearises the vector field.
A hyperbolic equilibrium, with no eigenvalues on the imaginary axis,
is dynamically unstable and on a sufficiently small neighbourhood the
motion is completely determined by the linearisation.

In the elliptic case the non-linear terms cannot be disposed of
completely, but lead to normal forms of which one hopes that they
capture the essence of the dynamics.
The reasons for irremovable terms are the resonances between the
eigenvalues on the imaginary axis.
Excluding zero eigenvalues, the resonances of lowest order, i.e.\
the $1{:}1$ and $1{:}{-}1$~resonances, relate double pairs of imaginary
eigenvalues.

\subsection{Resonant equilibria}\label{sec:resequi}
In the present paper we concentrate on the $1{:}1$~resonance and study
an equilibrium around which the Hamiltonian expands as
\begin{equation}
   H(q, p)  \;\; = \;\;
   \frac{1}{2} (p_1^2 + p_2^2) \, + \,
   \frac{1}{2} (q_1^2 + q_2^2) \; + \;
   \ldots
\label{eq:quadpartplus}
\end{equation}
where we omit the irrelevant constant term.
The Hessian $D^2H(0)$ is positive definite and this excludes
nilpotent terms.
Thus a $1{:}1$~resonance is always semisimple.
It occurs persistently in $3$--parameter families,
cf.~\cite{kummer1976, cr1998, duistermaat1984, cotter1986, hlr2003}.
This is in sharp contrast with the $1{:}{-}1$~resonance where the
Hessian is not definite.
Then we have to distinguish a semisimple and a non-semisimple case.
Unfolding the latter leads to the Hamiltonian Hopf bifurcation, which
occurs persistently in $1$--parameter families,
cf.~\cite{meer1985, cb2015, hm2003}.
The semisimple $1{:}{-}1$~resonance also occurs persistently in
$3$--parameter families, cf.~\cite{kummer1978, hlr2003, hhxxxx}.
We expect its unfolding to share features of that of the $1{:}1$~resonance.

A comprehensive study of $k{:}l$~resonances, excluding the
$1{:}{\pm} 1$~cases, has been made in~\cite{duistermaat1984}.
It turns out that all higher order cases are very similar to each
other.
In general the unfolding co-dimension of the unfolding is two, where
one parameter can be considered as a \emph{detuning} of the resonance
and the other is a \emph{modulus}, see~\cite{sanders1978, hanssmann2007}.
Exceptions are the resonances $1{:}2$ and~$1{:}3$ with co-dimensions
$1$ and~$3$, respectively.
Again one of the parameters is a detuning and in the case of
$1{:}3$~resonance, two parameters are moduli.
In all cases there is a bifurcation associated to the resonance.
In general a pair of stable and unstable periodic solutions branches
off from the origin.
The $1{:}2$ and $1{:}3$~cases have a slightly different unfolding
scenario, see~\cite{duistermaat1984, bhlv2003, hanssmann2007, cdhs2007}.
As mentioned before the non-semisimple or nilpotent
$1{:}{-}1$~resonance shows a different bifurcation (the Hamiltonian
Hopf bifurcation, see~\cite{meer1985}) and the bifurcations
triggered by the semisimple $1{:}{\pm} 1$~resonances are still open.

This paper is organized as follows.
In section~\ref{sec:informal} we state an informal version of our
main theorem.
Although informal it still contains the essential properties of the
main theorem.
Before proving our main result we review some facts on Hamiltonian
systems in section~\ref{sec:hamfacts}.
The system we study is in normal form and we discuss the
properties we use in section~\ref{sec:symm}, especially the induced
$S^1$--symmetry.
Finally in section~\ref{sec:uniunfo} we state and in
section~\ref{sec:proof} we prove our main theorem using singularity
theory for $S^1$--equivariant mappings.
The concluding section~\ref{sec:disco} puts our results in context.
Our approach fits in the tradition of
\cite{duistermaat1984, meer1985, cr1998} and it
complements~\cite{cotter1986}.

\subsection{Informal statement of the main theorem}
\label{sec:informal}
In order to state our main result we need a few definitions.
Here our aim is not full generality, the main theorem is formulated
more precisely in section~\ref{sec:mainthe}.

We study a $\cinf$~Hamiltonian system on~$\R^4$ with standard
symplectic form in the neighbourhood of an elliptic equilibrium in
$1{:}1$~resonance.
We may assume that the equilibrium is at the origin, thus the linear
part of the Hamiltonian $H$ at~$0$ vanishes.
The matrix associated to the linearisation of the Hamiltonian vector
field has coinciding pairs of eigenvalues with equal
\emph{symplectic sign}, therefore this matrix has no nilpotent part,
see~\cite{hoveijn1996a}.
As a consequence the quadratic part of the Hamiltonian in the
$1{:}1$~case has Morse index~$0$.
This contrasts with the $1{:}{-}1$~resonance where the corresponding
matrix generically does have a nilpotent part, see~\cite{meer1985}.

As a first step we apply several (symplectic) co-ordinate
transformations.
The first of these takes the quadratic part $H_2$ of~$H$ into the
form presented in equation~\eqref{eq:quadpartplus}.
Moreover, after a finite number of \emph{normal form} transformations
(see for example~\cite{meer1985}), we may assume that a corresponding
part of the Taylor expansion of~$H$ Poisson commutes with~$H_2$.
We now make an approximation by restricting to this finite part and
call it $H$ again.
The flow of~$H_2$ generates an $S^1$ symmetry group and the fact that
$H$ and~$H_2$ Poisson commute implies that $H$ is $S^1$--symmetric.
The consequences of this approximation are discussed in the
remarks following theorem~\ref{the:informal}.

The second step is a reduction with respect to the $S^1$ symmetry.
Restricted to the $3$--sphere $\{ H_2=1 \}$, the projection mapping
involved is a Hopf mapping so the \emph{reduced phase space} is a
$2$--sphere.
Then we apply \emph{equivariant singularity theory} to the
\emph{map germ} $(H,H_2)$ and find a \emph{universal unfolding}
subject to non-degeneracy conditions on the coefficients in the
higher order terms of~$H$.
By the nature of our method, we can not hope for more than local
results and we exploit this fact by switching to germs,
see~\cite{bl1975, martinet1982, montaldi2013}.
Very briefly: a map germ is the collection of mappings equal to one
another on an arbitrary small neighbourhood of a given point,
say~$0$.
Map germs are essentially determined by their Taylor expansions or
even Taylor polynomials in a sense that is made more precise in
section~\ref{sec:mapstab}.
In the sequel we say mapping but tacitly assume map germ.

In order to proceed we need the \emph{generators} of the
$S^1$--invariant functions as co-ordinates.
These are given by
\begin{align*}
   I_1 &= \halfje(q_1^2 + p_1^2 + q_2^2 + p_2^2) \\
   I_2 &= \halfje(q_1^2 + p_1^2 - q_2^2 - p_2^2) \\
   I_3 &= q_1 q_2 + p_1 p_2 \\
   I_4 &= q_1 p_2 - q_2 p_1 ,
\end{align*}
see section~\ref{sec:11symm} for more details.
The generators are not independent but related by the \emph{syzygy}
$I_1^2 = I_2^2 + I_3^2 + I_4^2$.
Nevertheless, $H$ and $H_2$ can now be expressed as functions
of~$I$, that is $H_2(I) = I_1$ and
$H(I) = H_2(I) + H_4(I) + H_6(I) + \cdots + H_k(I)$.
The final result is given in the next theorem.

\begin{theorem}\label{the:informal}
A universal unfolding of the $S^1$--invariant Hamiltonian
\begin{equation*}
   H(I) = I_1 + a_1 I_2^2 + a_2 I_3^2 + a_3 I_4^2
   + b_1 I_2^3 + b_2 I_3^3 + b_3 I_4^3 
\end{equation*}
is given by the five parameter family ($\mu \in \R^5$)
\begin{align*}
   H(I;\mu) = & I_1 + a_1 I_2^2 + a_2 I_3^2 + a_3 I_4^2
   + b_1 I_2^3 + b_2 I_3^3 + b_3 I_4^3\\
   & + \mu_1 I_2 + \mu_2 I_3 + \mu_3 I_4 + \mu_4 I_2^3 + \mu_5 I_3^3
\end{align*}
provided that the real coefficients $a_1$, $a_2$, $a_3$, $b_1$,
$b_2$ and $b_3$ satisfy the non-degeneracy condition
\begin{equation*}
   (a_1 - a_2)(a_2 - a_3)(a_3 - a_1) \, b_1 b_2 b_3 \neq 0.
\end{equation*}
\end{theorem}

This theorem holds for $S^1$--symmetric Hamiltonian systems in
$1{:}1$~resonance.
Let us make a few remarks on its scope.

\begin{remark}\label{rem:scope}\strut
\begin{enumerate}[topsep=0ex,partopsep=0px,itemsep=0px]
\item
The unfolding terms $\mu_4 I_2^3$ and $\mu_5 I_3^3$ can be replaced
by any pair from $I_2$, $I_3$ and $I_4$.
\item
The reduction of the $3$--sphere defined by
$\halfje(q_1^2 + p_1^2 + q_2^2 + p_2^2) = h_2$ to the $2$--sphere
$I_2^2 + I_3^2 + I_4^2 = h_2^2$ is regular if $h_2 \neq 0$, so every
point on the reduced phase space corresponds to an $S^1$--orbit on
the original phase space~$\R^4$.
\item
On the reduced phase space the solution curves are defined by
$(H,H_2) = (h,h_2)$.
Thereby time parametrisation is lost.
Solution curves consisting of a single point on the reduced phase
space correspond to periodic orbits on~$\R^4$, whereas closed curves
on the reduced phase space correspond to $2$--tori on~$\R^4$.
The former are generically isolated on~$S^2$, but the latter come in
$1$--parameter families.
\item\label{itm:flat}
Non--$S^1$--symmetric perturbations (i.e.\ including
non--$S^1$--invariant terms in the Taylor expansion of~$H$) do
affect our result.
However, normal form transformations enable us to make these
perturbations as small as we wish.
Nevertheless their effect is that families of $2$--tori,
on~$\R^4$, do not survive as such.
From {\sc kam}~theory one expects that these families are
\emph{Cantorised}, i.e.\ the $2$--tori persist as a Cantor
subfamily of large $2$--dimensional Hausdorff measure, where the
dense set of internal resonances leads to gaps in the
parametrisation.
Periodic orbits, as long as they are elliptic or hyperbolic, do
persist, as do their bifurcations.
Thus our result gives information on low periodic orbits of general
Hamiltonian systems in $1{:}1$~resonance.
Homoclinic and heteroclinic connections on the reduced phase space
generically do break up under non--$S^1$--symmetric perturbations
yielding chaotic regions familiar from Poincar\'e sections of for
example the H\'enon--Heiles system.
\item
In view of the previous remark, the bifurcation diagram for the
equilibrium at $0$ on~$\R^4$ with branches of periodic orbits is
valid for general Hamiltonian systems in $1{:}1$~resonance.
\end{enumerate}
\end{remark}

\section{A few facts about Hamiltonian systems}\label{sec:hamfacts}
Here we very briefly review some facts from the theory of Hamiltonian
systems.
We concentrate on~$\R^4$.
However everywhere in the following sections $\R^4$ can be replaced
by~$M$, a $\cinf$ real symplectic manifold.
For a thorough treatment we refer to for example~\cite{am1987, arnold1980}.

\subsection{Symplectic spaces and Hamiltonian systems}
Let $\omega$ be a closed, non-degenerate skew symmetric $2$--form
on~$\R^4$, making $(\R^4, \omega)$ a \emph{symplectic space}.
Furthermore let $H$ be a function in $\cinf(\R^4, \R)$, then
the triple $(\R^4, \omega, H)$ is called a smooth real
\emph{Hamiltonian system}.
Now let $\cX(\R^4)$ be the set of smooth vector fields on~$\R^4$.
The vector field $\X{H} \in \cX(\R^4)$ satisfying
\begin{equation*}
   \omega(\X{H},Y) = dH(Y)
\end{equation*}
for all $Y \in \cX(\R^4)$, is called the \emph{Hamiltonian vector field}
of~$H$.
The vector field $\X{H}$ defines the \emph{flow} of the Hamiltonian
system on~$\R^4$, we also call this the flow of~$H$.
A function~$f$ is preserved under the flow of the vector field~$\X{H}$
if and only if the Lie derivative of~$f$ is identically zero.
Using $L_{X_H}(f) = df(\X{H})$ we find that the Hamiltonian
function~$H$ is preserved by the flow of $\X{H}$ because
\begin{equation*}
   L_{X_H}(H) = dH(\X{H}) = \omega(\X{H},\X{H}) = 0.
\end{equation*}
The last equality follows from the skew symmetry of~$\omega$.

\subsection{Poisson brackets}

Let $f$ and $g$ be in $C^{\infty}(\R^4,\R)$, then we define the Poisson
bracket of $f$ and~$g$ as
\begin{equation*}
   \poisson{f}{g} = \omega(\X{f},\X{g}).
\end{equation*}
It follows from this definition that
\begin{equation*}
   \poisson{f}{g} = L_{X_g}(f) = -L_{X_f}(g).
\end{equation*}
Suppose that the function~$f$ is preserved under the flow of~$\X{H}$,
then
\begin{equation*}
   0 = L_{X_H}(f) = \poisson{f}{H}
\end{equation*}
and vice verse, so once we have the Poisson bracket we do not need
the vector field~$\X{H}$ to determine whether $f$ is preserved
under the flow of~$H$.
Furthermore $\poisson{f}{g} = - \poisson{g}{f}$ so
$\poisson{H}{H} = 0$ from which again follows that $H$ is preserved
under the flow of~$\X{H}$.
The Poisson bracket satisfies Jacobi's identity whence Hamiltonian
vector fields form a Lie algebra; in fact we have
\begin{equation*}
   [\X{f}, \X{g}] = - \X{\poisson{f}{g}}.
\end{equation*}
Thus $(\cinf(\R^4),\poisson{\cdot}{\cdot})$ is a Lie algebra of functions.

\subsection{Standard forms}

Darboux's theorem now states that there are co-ordinates such that
$\omega$ becomes constant.
Then by applying linear algebra we can bring $\omega$ into a standard
form such that
\begin{equation*}
   \omega(\xi,\eta) = \inprod{\xi}{\Omega \eta}
\end{equation*}
for all $\xi, \eta \in \R^4$.
Here $\inprod{\cdot}{\cdot}$ is the standard inner product on~$\R^4$
and $\Omega$ is a linear mapping with
$\Omega = -\Omega^t = -\Omega^{-1}$ which takes the standard form
\begin{equation*}
   \Omega = \left(\begin{array}{rr}0&I\\-I&0\end{array}\right)
\end{equation*}
on the standard basis $\basis{e_1, e_2, f_1, f_2}$.
Let us take co-ordinates $z = (q_1, q_2, p_1, p_2)$ with respect to
this basis, then the Poisson bracket becomes
\begin{equation*}
   \poisson{f}{g} = \sum^2_{i=1} \left(
   \partieel{f}{q_i}\partieel{g}{p_i} -
   \partieel{f}{p_i}\partieel{g}{q_i}
   \right).
\end{equation*}
Using the Poisson bracket on these co-ordinates we obtain the canonical
equations of motion
\begin{equation*}
   \dot{q}_i = \poisson{q_i}{H} = \partieel{H}{p_i}, \;\;
   \dot{p}_i = \poisson{p_i}{H} = -\partieel{H}{q_i}
\end{equation*}
for the Hamiltonian~$H$.
The Poisson bracket allows us to use functions instead of vector
fields, which simplifies many computations.

\section{Resonant Hamiltonian systems and $S^1$--symmetry}
\label{sec:symm}
On the symplectic space $(\R^4, \omega)$ we consider
$\cinf$~Hamiltonian systems with an equilibrium at the origin.
Furthermore suppose that the linearisation of the corresponding
Hamiltonian vector field has resonant imaginary eigenvalues.

When this system has been transformed into normal form it admits an
$S^1$--symmetry group.
Resonant eigenvalues are not generic, but when they appear in
parameter families of Hamiltonian systems they are a source of
bifurcations.
Therefore it is useful to study unfoldings of resonant systems.
Most resonances in $4$--dimensional Hamiltonian systems have been
studied before, see~\cite{duistermaat1984} and references therein.
This approach has to be refined for the $1{:}1$ and
$1{:}{-}1$~resonances, where the sign is the \emph{symplectic sign}.
See~\cite{meer1985} for an extensive study of the so-called nilpotent
$1{:}{-}1$~resonance which in a parameter family gives rise to the
\emph{Hamiltonian Hopf bifurcation}.
Our aim here is to study the $1{:}1$ resonance.
While this case has already been considered in~\cite{cotter1986}, the
arguments presented there are incomplete. 

A resonant Hamiltonian system naturally leads to an $S^1$--invariant
system when passing to a normal form truncation.
But we may also consider Hamiltonian systems with an externally given
symplectic $S^1$--action.
Our results hold for such systems as well, provided that the
$S^1$--action satisfies the conditions in the next section.

\subsection{$S^1$--symmetry related to the $1{:}1$ resonance}
\label{sec:11symm}
Since we work in the class $\cinf(\R^4)$ the Hamiltonian function~$H$
has an infinite Taylor series.
We now put some more structure on these functions by collecting
homogeneous terms, turning $(\cinf(\R^4),\poisson{\cdot}{\cdot})$
into a graded Lie algebra.
Then we expand
\begin{equation*}
   H = H_2 + H_3 + \cdots + H_k + \cdots
\end{equation*}
with $H_k \in \R[z]$ homogeneous of degree~$k$.
The normal form procedure acts in a very nice way on this Lie algebra,
for details see~\cite{meer1985}.
The final result is that for the normal form we have
$\poisson{H_2}{H_k}=0$ for all $k$ and therefore $\poisson{H_2}{H}=0$.
This means that the normal form of~$H$ is invariant under the flow
of~$H_2$ which is generated by~$\X{H_2}$.
Now we assume that the linear part $\X{H_2}$ of the vector
field~$\X{H}$ is in $1{:}1$ resonance, then (the normal form of) $H$
is $S^1$--invariant with respect to the $S^1$--action
\begin{equation}
\label{eq:s1actie}
\begin{array}{cccc}
   \phi \; : & S^1 \times \R^4 & \longrightarrow & \R^4  \\
   & (\varphi, z) & \mapsto & R_{\varphi} z
\end{array}
\end{equation}
where
\begin{equation*}
   R_{\varphi} = \left(
   \begin{array}{cccc}
      \cos \varphi & -\sin \varphi & 0 &  0\\
      \sin \varphi & \hpm \cos \varphi & 0 &  0\\
      0 &  0 & \cos \varphi & -\sin \varphi\\
      0 &  0 & \sin \varphi & \hpm \cos \varphi\\
   \end{array}\right)
\end{equation*}
and $z=(q_1,p_1,q_2,p_2)$.
The quadratic part of such a Hamiltonian systems reads
\begin{equation*}
   H_2(q_1,p_1,q_2,p_2) =
   \halfje(q_1^2 + p_1^2) + \halfje(q_2^2 + p_2^2).
\end{equation*}
Note that this function has Morse index~$0$ which is intimately
related to the fact that the eigenvalues of the linear part of the
corresponding Hamiltonian vector field have equal symplectic sign,
see~\cite{bc1977}.

Every $S^1$--invariant $\cinf$--function can be written as a function
of so called \emph{invariants}.
This is a consequence of far more general results which we now state.
We start with a theorem on \emph{invariant polynomials}.

\begin{theorem}[Hilbert, Schwartz]\label{the:schwartz}
Let $\Gamma$ be a compact group which acts linearly on~$\R^n$ and let
$\R[z]^{\Gamma}$ denote the set of $\Gamma$--invariant polynomials.
Then a finite number $r$ of polynomials
$\rho_1, \ldots, \rho_r \in \R[z]^{\Gamma}$ exist that
generate~$\R[z]^{\Gamma}$.
The $\rho_1, \ldots, \rho_r$ form a \emph{Hilbert basis} and are
called \emph{generators}.
Furthermore every $\Gamma$--invariant $\cinf$--function
$f \in \cinf(\R^n)^{\Gamma}$ can be written as a $\cinf$--function
$\hat{f} \in \cinf(\R^r)$ of the $r$~generators of~$\R[z]^{\Gamma}$.
\end{theorem}

Unfortunately the function $\hat{f}$ need not be unique for there may
be syzygies among the~$\rho_j$.

Let us now determine the invariants of the $S^1$--action associated
to the $1{:}1$~resonance.
These are polynomials on the phase space and they Poisson commute
with~$H_2$.

\begin{lemma}\label{lem:rgens}
The generators in $\R[q, p]^{S^1}$ of the invariants of the
$S^1$--action associated to the {\rm $1{:}1$}~resonance are given by
\begin{align*}
   I_1 &= \halfje(q_1^2 + p_1^2 + q_2^2 + p_2^2) \\
   I_2 &= \halfje(q_1^2 + p_1^2 - q_2^2 - p_2^2) \\
   I_3 &= q_1 q_2 + p_1 p_2 \\
   I_4 &= q_1 p_2 - q_2 p_1
\end{align*}
with \emph{syzygy} $I_1^2 = I_2^2 + I_3^2 + I_4^2$.
\end{lemma}

For a proof we refer to~\cite{cb2015}.

Thus every $S^1$--invariant $\cinf$--function on~$\R^4$ can be written
as a $\cinf$--function of the Hilbert basis $\{ I_1, I_2, I_3, I_4
\}$.  From now on we restrict to a smaller set of functions, namely
the formal series in~$\Ri$. The reasons we can do this are 1) every
polynomial in~$I$ is the Taylor series of a $\cinf$-function of~$I$;
2) we only allow for a finite number of conditions on the coefficients
of a series.  The latter means that we do not encounter the subtleties
on infinitely flat functions, however see remark~\ref{rem:scope},
item~\ref{itm:flat}. Moreover we are only interested in
$\cinf$--functions that are zero at the origin. Therefore we only
consider formal series without constant terms, denoted by $\Ri_0$.

Now a function in~$\Ri_0$ is not unique, due to the syzygy among
the generators.
In this respect it is worth noting that when we consider functions
in~$\Ri_0$ modulo the ideal generated by
$I_1^2 - (I_2^2 + I_3^2 + I_4^2)$, denoted by~$\Ris$, we have the
following splitting, see~\cite{cb2015}.
This splitting is also not unique, but seems natural in view of the
syzygy.

\begin{lemma}\label{lem:modrel}
\quad
${\displaystyle \Retdv/{\sim} \;\; = \;\; \Rtdv \; \oplus \; I_1 \Rtdv}$.
\end{lemma}

When chosen in this last space the function~$\hat{f}$ in
theorem~\ref{the:schwartz} is unique.
Now that we know the generators of the invariants we can write $H$
and~$H_2$ as functions of these.
In particular, we have $H_2(I) = I_1$.

\subsection{Reduction of the $S^1$--symmetry:
            Hamiltonian systems on~$S^2$}
\label{sec:redsphere}
We are primarily interested in the flow of~$H$.
Since the flow of~$H$ and the $S^1$--action commute
($H$ and~$H_2$ Poisson commute), the orbits of $H$ through an
$S^1$--orbit are equivalent.
Therefore we wish to reduce to the orbit space $\R^4 / S^1$ where
points correspond to $S^1$--orbits on~$\R^4$.
The projection mapping 
\begin{equation}
\label{hilbertmapping}
   (q,p) \mapsto I
\end{equation}
defined in lemma~\ref{lem:rgens} just does that.
It allows us to reduce the dynamics of $H$ on~$\R^4$ to a
$2$--dimensional phase space.

The $S^1$--action is generated by the vector field~$\X{H_2}$.
Now $H_2$ is preserved by its own flow, therefore the $S^1$--action
preserves $I_1$ which defines a $3$--sphere
\begin{equation*}
   \left\{ \; (q, p \in \R^4 \mitdE
   h_2 = I_1 = \halfje(q_1^2 + p_1^2 + q_2^2 + p_2^2) \; \right\}
   \enspace .
\end{equation*}
As $H$ and~$H_2$ Poisson commute, the flow of~$H$ also preserves
this $3$--sphere.
Because of the syzygy $I_1^2 = I_2^2 + I_3^2 + I_4^2$ the projection
mapping takes the flow of~$H$ to a $2$--sphere in the reduced phase
space; the reduced phase space is determined by $I_1 = h_2$,
$I_1^2 = I_2^2 + I_3^2 + I_4^2$.
The reduced dynamics of~$H$ can simply be characterised by the level
$h$ of~$H$.
This means that an orbit of the reduced flow of~$H$ is determined by
the equations
\begin{align*}
   (H, H_2) &= (h, h_2) \\
   I_2^2 + I_3^2 + I_4^2 &= h_2^2
   \enspace .
\end{align*}
The reduced dynamics of~$H$ consists of curves on a $2$--sphere.
Note that in order to know the time parametrisation of these curves
we still have to solve a generally difficult differential equation.
But we do have a full geometric characterisation.

This leads us to the following.
We consider the set of smooth $S^1$--invariant mappings
$\cinf(\R^4,\R^2)^{S^1}$ of the form $(H, H_2)$.
The reduced dynamics of~$H$ is determined once we specify its
value by $(H, H_2) = (h, h_2)$.
In the next section we address the question whether a
polynomial~$H$ exists such that this mapping is stable in the
sense of singularity theory.

\begin{remark}\strut
\begin{enumerate}[topsep=0ex,partopsep=0px,itemsep=0px]
\item
For a far more complete account of general \emph{regular reduction}
see for example~\cite{am1987, arnold1980}.
More details about the $1{:}1$~resonance can be found
in~\cite{cb2015} where the projection mapping~\eqref{hilbertmapping}
is shown to be the Hopf mapping from $S^3$ to~$S^2$.
\item
Other resonances like $k{:}l$ give rise to a different reduced phase
space, having singularities.
These arise from non-trivial isotropy subgroups of the
$S^1$--symmetry group in these cases.
They again turn up in new generators with a higher order syzygy.
In $4$--dimensional resonant Hamiltonian systems the situation is
relatively simple, there are four generators and one syzygy.
In higher dimensions both the number of generators and the number of
syzygies depend on the resonance, {\it i.e.}\ on the ratios
$k_1: k_2 : \cdots : k_n$, making it computationally difficult.
Then the Gr\"obner basis algorithm is indispensable.
\end{enumerate}
\end{remark}

Both sides of the syzygy define a \emph{Casimir element}, i.e.\
their Poisson brackets with the~$I$ vanish.
A straightforward calculation yields table~\ref{tab:pbrg} of Poisson
brackets.

\begin{table}[hbtp]
\begin{equation*}
\begin{array}{l|rrrr}
\{\cdot,\cdot\} & I_1 & I_2  & I_3  & I_4  \\\hline
I_1             & 0   & 0    & 0    & 0    \\
I_2             & 0   & 0    &-2I_4 & 2I_3 \\
I_3             & 0   & 2I_4 & 0    &-2I_2 \\
I_4             & 0   &-2I_3 & 2I_2 & 0    \\
\end{array}
\end{equation*}
\caption{\textit{
Poisson bracket of the real generators $I$ of the invariants.
\label{tab:pbrg}}}
\end{table}

The invariants from lemma~\ref{lem:rgens} are sometimes called Hopf
variables.
Indeed, $I_1$ generates the $S^1$--symmetry~\eqref{eq:s1actie} and
hence is an integral of motion for every Hamiltonian system with
that symmetry.
The Hopf mapping
\begin{equation*}
\begin{array}{cccc}
   (I_2, I_3, I_4) : & S_{2 I_1}^3 \longrightarrow  & S_{I_1^2}^2
\end{array}
\end{equation*}
from the $3$--sphere
\begin{equation*}
   S_{2 I_1}^3  \;\; = \;\; \left\{ \;  (q, p) \in T^*\R^2
   \mitdE  q_1^2 + q_2^2 + p_1^2 + p_2^2 = 2 I_1  \; \right\}
\end{equation*}
to the $2$--sphere
\begin{equation*}
   S_{I_1^2}^2  \;\; = \;\; \left\{ \;  (I_2, I_3, I_4) \in \R^3
   \mitdE  I_2^2 + I_3^2 + I_4^2 = I_1^2  \; \right\}
\end{equation*}
performs the reduction to one degree of freedom by identifying
points related through~\eqref{eq:s1actie}.

The phase portraits are obtained by intersecting, within~$\R^3$, the
level sets of the Hamiltonian $H = H(I_2, I_3, I_4)$ with~$S^2$.
Where $H$ is a Morse function, this yields finitely many centres and
saddles, with generically no heteroclinic connections between the
latter.
Under variation of parameters local and global bifurcations may
occur.

\section{The universal unfolding}\label{sec:uniunfo}
In this section we state our main theorem.
First we provide a context for the theorem by introducing the notion
of stable mappings under left-right-equivalence.

\subsection{Equivalence classes for $S^1$--invariant
            Hamiltonian systems}
\label{sec:aequiv}
The meaning of `universal unfolding' depends on the universe in which
we work and the notion of equivalence.
As explained in section~\ref{sec:redsphere} we consider Hamiltonian
systems on~$\R^4$ that are $S^1$--invariant and can be reduced to~$S^2$.
If we content ourselves with characterising the reduced dynamics of~$H$
by the orbits only we just need to specify values of $H$ and~$H_2$.
That is the orbits of the reduced Hamiltonian systems are the fibres of
the mapping $(H, H_2)$.
Note however that $H_2(I) = I_1$ and $H_2$ is an integral of the
Hamiltonian system.
So~$I_1$ is constant and therefore not to be considered as a variable
but rather a parameter.
Furthermore note that the fibres of the mappings $(H, H_2)$
and $(H, H_2^2)$ are identical.
Using the relation of the generators of the invariants~$I$, we have
$H_2(I)^2 = I_1^2 = I_2^2 + I_3^2 + I_4^2$.
This leads us to define $K(I) = H_2(I)^2$ and consider the mapping
$\cF(I) = (H(I), K(I))$ on our universe $\cinf(\R^4, \R^2)^{S^1}_0$,
the $S^1$--invariant $\cinf$--mappings from $\R^4$ to~$\R^2$ taking
$(0,0)$ to $(0,0)$.

A natural notion of equivalence on $\cinf(\R^4, \R^2)^{S^1}_0$ is
provided by so called \emph{left-right-equivalences}, see
definition~\ref{def:lraequiv} below.
For if $\cF$ and~$\cG$ are left-right-equivalent then the fibres
of $\cF$ and~$\cG$ are diffeomorphic.
This in turn implies that the orbits of the $S^1$--invariant
Hamiltonian systems in $\cF=(H, K)$ and $\cG=(H', K')$ can be
mapped to each other by a simple diffeomorphism.

\begin{definition}\label{def:lraequiv}
The mappings $\cF, \cG \in \cinf(\R^4, \R^2)^{S^1}_0$ are called
\emph{left-right-equivalent} if $(\psi, \phi) \in \Diff(\R^2)_0 \times
\Diff(\R^4)^{S^1}_0$ exists such that $(\psi, \phi) \cdot \cF = \cG$,
where $(\psi, \phi) \cdot \cF = \psi \circ \cF \circ \phi$.
\end{definition}

\subsubsection{Stable $S^1$--invariant mappings,
               co-dimension and unfolding}
\label{sec:mapstab}
The idea of stability of a mapping~$\cF$ is that every mapping $\cG$
nearby~$\cF$ is equivalent to~$\cF$, or put differently, that $\cG$
is an element of the orbit of~$\cF$ under left-right-equivalence.
Here we give a short overview in a series of definitions and
theorems.
\begin{definition}\label{def:lrorbit}
The \emph{orbit} of $\cF \in \cinf(\R^4, \R^2)^{S^1}_0$ under
left-right-equivalences is given by
\begin{equation*}
   Orb_{\cF} = \{(\psi, \phi) \cdot \cF \;|\;
   (\psi, \phi) \in \Diff(\R^2)_0 \times \Diff(\R^4)^{S^1}_0\}.
\end{equation*}
\end{definition}
To define `nearby' we use the definition of a deformation.
\begin{definition}\label{def:defo}
A \emph{deformation (or unfolding)} of a mapping
$\cF \in \cinf(\R^4, \R^2)^{S^1}_0$ is a $\cinf$--mapping
$\cF : \R^4 \times \R^p \longrightarrow \R^2$ defining
a family of $S^1$--equivariant $\cF_{\nu}$, $\nu \in \R^p$,
such that $\cF_0 = \cF$.
\end{definition}
This allows to formulate a parametric version of $\cF$ being an
interior point of the orbit of~$\cF$.
\begin{definition}\label{the:stab}
A mapping $\cF \in \cinf(\R^4, \R^2)^{S^1}_0$ is called \emph{stable} if
for every deformation $\cF_{\nu}$ there is an open neighbourhood $U$
of $0 \in \R^p$ such that for all $\nu \in U$, $\cF_{\nu} \in Orb_{\cF}$.
\end{definition}
The conditions of stability in this sense are hard to check.
The conditions of \emph{infinitesimal stability} are much easier to
check and this notion of stability turns out to be equivalent with
the previous one.
\begin{definition}\label{def:infistab}
$\cF$ is called \emph{infinitesimally stable} if the tangent space of
$Orb_{\cF}$ at~$\cF$ is equal to the tangent space of
$\cinf(\R^4, \R^2)^{S^1}_0$ at~$\cF$.
\end{definition}
A proof of the next theorem can be found in \cite{martinet1982}.
\begin{theorem}\label{the:infistab}
A mapping is stable if and only if it is infinitesimally stable.
\end{theorem}
Stable mappings form an open and dense subset of
$\cinf(\R^4, \R^2)^{S^1}_0$, see~\cite{poenaru1976}.
A mapping that fails to be stable has therefore non-zero
co-dimension
\begin{definition}\label{def:equivfamilies}
Two deformations $\cF_{\nu}$ and $\cG_{\mu}$ are
{\it left-right-equivalent} if there are
$(\psi_{\nu}, \phi_{\nu})$ and $\mu(\nu)$ with
$\psi_{\nu} \circ \cF_{\nu} \circ \phi_{\nu} = \cG_{\mu(\nu)}$.
\end{definition}
This allows to generalize the previous discussion of mappings to
deformations.
\begin{definition}\label{def:versal}
A versal unfolding is a stable deformation.
\end{definition}
The minimal number of parameters of a versal unfolding of a
mapping~$\cF_0$ coincides for $\cinf(\R^4, \R^2)^{S^1}_0$ with the
co-dimension of~$\cF_0$.

\subsubsection{The tangent space of $Orb_{\cF}$ at $(H, K)$}
\label{sec:tan}
Let $\cX(\R^4)$ be the Lie algebra of $\Diff(\R^4)_0$ and
$\cX(\R^4)^{S^1}$ be the Lie algebra of $\Diff(\R^4)^{S^1}_0$.
\begin{lemma}\label{lem:tan}
The tangent space of $Orb_{\cF}$ of $\cF \in \cinf(\R^4, \R^2)^{S^1}_0$ at
$(H, K)$ is given by
\begin{equation*}
   \{ X(\cF) + d\cF(Y) \;|\; X \in \cX(\R^2),\; Y \in \cX(\R^4)^{S^1} \}.
\end{equation*}
\end{lemma}
\begin{proof}
For every near-identity transformation
$(\psi, \phi) \in \Diff(\R^2)_0 \times \Diff(\R^4)^{S^1}_0$
there exist $X \in \cX(\R^2)$ and $Y \in \cX(\R^4)^{S^1}$
such that for some $t \in R$ we have
$(\psi, \phi) = (\e^{t X}, \e^{t Y})$.
Then the tangent vectors are
$\frac{d}{dt}(\e^{t X} \circ \cF \circ \e^{t Y})|_{t=0} = X(\cF) + d\cF(Y)$.
\end{proof}
Taking a closer look at the tangent space of $Orb_{\cF}$ at $\cF = (H, K)$
in lemma~\ref{lem:tan}; we explicitly have
\begin{equation}\label{eq:tan}
   X(\cF) + d\cF(Y) = (X_1(H, K) + Y(H), X_2(H, K) + Y(K)).
\end{equation}
In this expression $X$ is any vector field on $\R^2$, but $Y$ is
an $S^1$--equivariant vector field on~$\R^4$.
Using theorem~\ref{the:infistab} we have to check that every
$S^1$--equivariant map germ can be written as
$(X_1(H, K) + Y(H), X_2(H, K) + Y(K))$
for a suitable choice of $X$ and~$Y$.

\subsubsection{The restricted tangent space of $Orb_{\cF}$ at $(H, K)$}
\label{sec:rtan}
The $S^1$--equivariant vector fields are such that $Y(K)$ can be any
function of degree 2 and higher in the set of $S^1$--invariant
functions on~$\R^4$.
This follows from an explicit calculation of these vector fields in
section~\ref{sec:vvs}.
Thus the stability of~$\cF$ is determined by the first component.
More precisely we have the following.
\begin{proposition}\label{pro:restriction}
The co-dimension of $(H, K)$ in $\cinf(\R^4, \R^2)^{S^1}_0$ with the
full group of left-right-equivalences is equal to the co-dimension
of~$H$ in $\cinf(\R^4)^{S^1}_0$ with the group of
left-right-equivalences that fix~$K$.
\end{proposition}
Therefore we restrict to vector fields in $Y \in \cX(\R^4)^{S^1}$
such that $X_2(H, K) + Y(K) = 0$.
Or, from a slightly different point of view, we look for a normal
form of the mapping $\cF = (H, K)$.
But the second component can already be regarded as being in normal
form.
Therefore we may restrict to transformations that preserve~$K$,
that is $X_2(H, K) + Y(K) = 0$.
\begin{lemma}\label{lem:aequivxfix}
The set of $S^1$--equivariant vector fields~$Y$ with
$Y(K) \in \Rhk$ can be decomposed as the direct sum of
two modules.
The first is a module over~$\Ris$ and consists of vector fields
$Y \in \cX(\R^4)^{S^1}$ taking $K$ to zero.
The second is a module over $\Rhk$, generated by vector
fields $Y \in \cX(\R^4)^{S^1}$ taking $K$ to $K$ or to~$H$.
\end{lemma}
\begin{proof}
$S^1$--equivariant vector fields~$Y$ such that
$Y(K) \in \Rhk$ are generated by $S^1$--equivariant
vector fields satisfying one of the three equations $Y(K) = 0$,
$Y(K) = K$ and $Y(K) = H$.
\end{proof}
From now on we consider the restricted tangent space of~$Orb_H$ under
left-right-transformations and we call it~$T_1$.
The restricted tangent space of~$Orb_H$ is again the sum of two
(function) modules $\cJ \oplus \cM$.
Suppose $U_1,\cdots,U_k$ generate the solutions of $Y(K) = 0$ and
$V_1$ and $V_2$ solve $Y(K) = K$ and $Y(K) = H$, respectively.
Furthermore let $F_i = U_i(H)$ for $i \in \{1,\ldots,k\}$ and
$G_j = V_j(H)$ for $j \in \{1,2\}$.
Then we have the following.
\begin{lemma}\label{lem:orbh}
The restricted tangent space of $Orb_H$ is the sum of two modules
$\cJ \oplus \cM$, the first is a module over~$\Ris$ and the second
is a module over~$\Rhk$.
That is, every function $f$ in the tangent space of $Orb_H$ is of
the form
$f = \xi_1 F_1 + \cdots + \xi_k F_k + \eta_0 + \eta_1 G_1 + \eta_2 G_2$,
with $\xi_i \in \Ris$ and $\eta_i \in \Rhk$.
\end{lemma}

Thus the question about the co-dimension and universal unfolding of
the mapping $\cF \in \cinf(\R^4, \R^2)^{S^1}_0$ reduces to finding the
co-dimension and a complement of the first component of the tangent
space of~$\cF$ with respect to restricted left-right transformations.
This in turn can be reformulated as follows.
Let $\GGG$ be the mapping
\begin{equation*}
\begin{array}{cccc}
   \GGG \; : & \Big(\cinf(\R^4,\R)^{S^1}_0\Big)^k \times
   \Big(\cinf(\R^2,\R)\Big)^3_0 & \longrightarrow & \cinf(\R^4,\R)^{S^1}_0 \\
   & (\xi_1,\ldots,\xi_k,\eta_0,\eta_1,\eta_2) & \mapsto &
   X_1(H, K) + Y(H)
\end{array}
\enspace .
\end{equation*}
Then the questions we want to answer are:
\begin{enumerate}[topsep=0ex,partopsep=0px,itemsep=0px]
\item What is the co-dimension of the image of $\GGG$ in~$\Ris$~?
\item If the latter is nonzero, then what is a complement?
\end{enumerate}

\subsection{Statement of main theorem}\label{sec:mainthe}
Our main theorem is about the universal unfolding of the mapping
$(H, K): \cinf(\R^4)^{S^1}_0 \longrightarrow \R^2$ with respect to
restricted left-right-equivalence from the previous section.
That is we consider all left-right transformations that preserve~$K$.
As explained in section~\ref{sec:informal} we are interested in the
fibres of the mapping $(H, K)$.
For this mapping we have the following result.
\begin{theorem}\label{the:uniunfo}
The universal unfolding of the mapping $(H, K)$ with respect to
restricted left-right-equivalence is given by
\begin{equation*}
\begin{aligned}
   H(I;\mu) &= I_1 + a_1 I_2^2 + a_2 I_3^2 + a_3 I_4^2
   + b_1 I_2^3 + b_2 I_3^3 + b_3 I_4^3\\
   & + \mu_1 I_2 + \mu_2 I_3 + \mu_3 I_4 + \mu_4 I_2^3 + \mu_5 I_3^3\\
   K(I) &= I_2^2 + I_3^2 + I_4^2
\end{aligned}
\end{equation*}
provided that the real coefficients $a_1$, $a_2$, $a_3$ and $b_1$,
$b_2$ and $b_3$ satisfy the non-degeneracy condition
\begin{equation*}
   (a_1 - a_2)(a_2 - a_3)(a_3 - a_1) \, b_1 b_2 b_3 \neq 0.
\end{equation*}
The parameters $\mu_4$ and $\mu_5$ are moduli.
\end{theorem}

\section{Proof of main theorem}\label{sec:proof}
We now prove our main theorem.
Our starting point is the mapping $\cF = (H, K)$.
We split the higher order terms of~$H$ into two parts,
$H_4$ is of degree $2$ in~$I$, $H_6$ is of degree~$3$.

The proof consists of several steps which we now list.

\begin{enumerate}[1)]\itemsep 0pt
\item
Apply preliminary transformations to~$\cF$ to get rid of as many
coefficients as possible.
\item
Determine the tangent space of $Orb_{\cF}$ at~$\cF$.
\item
Find the $S^1$--equivariant vector fields on~$\R^4$.
\item
Observe that we can restrict to the first component of~$\cF$ using
restricted vector fields.
\item
Observe that we can proceed by degree when we split
$\cinf(\R^4,\R)^{S^1}_0$ as a direct sum of spaces of homogeneous
polynomials.
The cases of relative large degree turn out to be the easiest.
Then we are left with a finite number of low degree cases that
have to be treated separately.
\end{enumerate}

\subsection{Preliminary transformations}\label{sec:prels}
We start with the mapping $\cF = (H, K)$, where $H$ is a polynomial
of degree $3$ in $I$, that is $H = H_2 + H_4 + H_6$.
We assume that symplectic transformations already have been used
exhaustively.
But since we consider $\cF$ in a more general context, more
transformations are allowed.

The first observation is that we can always subtract $H_2$ from~$H$
because $H_2$ is a conserved function in the sense of Hamiltonian
systems.
Thus we have $H = H_4 + H_6$.
Furthermore, since $H_2(I) = I_1$ we consider $I_1$ as a parameter.
Therefore $I_1$ appears at most in the coefficients of~$H$.
So in fact $H$ and $K$ only depend on $I_2$, $I_3$ and~$I_4$,
without further restrictions or relations.
\begin{align*}
   K(I_2, I_3, I_4) &= I_2^2 + I_3^2 + I_4^2\\
   H_4(I_2, I_3, I_4) &= a_1 I_2^2 + a_2 I_3^2 + a_3 I_4^2
   + a_{23} I_2 I_3 + a_{24} I_2 I_4 + a_{34} I_3 I_4\\
   H_6(I_2, I_3, I_4) &= b_1 I_2^3 + b_2 I_3^3 + b_3 I_4^3
\end{align*}
The second observation is that by a transformation from
$id \times SO(3)$ we can always achieve $a_{23} = 0$, $a_{24} = 0$
and $a_{34} = 0$.
Note that such a transformation preserves both $K$ and the
relation $I_1^2 = I_2^2 + I_3^2 + I_4^2$.

\begin{remark}
We may include more third degree terms in $H_6$, like~$I_2 I_4^2$.
However, they turn out to be unimportant.
\end{remark}

\subsection{$S^1$--equivariant vector fields}\label{sec:vvs}
Considering the mapping $(H, K)$ instead of $(H, H_2)$ where $I_1$ is
a parameter, we take $I_2$, $I_3$ and $I_4$ as co-ordinates on $\R^3$
without any restrictions.
Now $(H, K)$ is a mapping in $\cinf(\R^3, \R^2)_0$.
Origin preserving transformations on $\R^3$ are generated by the
vector fields
\begin{equation}\label{eq:vfs}
\begin{aligned}
   X_1 &= I_2 \vvx, & X_2 &= I_3 \vvx, & X_3 &= I_4 \vvx,\\
   X_4 &= I_2 \vvy, & X_5 &= I_3 \vvy, & X_6 &= I_4 \vvy,\\
   X_7 &= I_2 \vvz, & X_8 &= I_3 \vvz, & X_9 &= I_4 \vvz.
\end{aligned}
\end{equation}
To define the restricted tangent space of the mapping $\cF$ we have to
find the vector fields solving $X(K) = 0$, $X(K) = K$ and $X(K) = H$.
\begin{lemma}\label{lem:vfs}
The vector fields solving $X(K) = 0$ are generated by
\begin{equation*}
   U_1 = X_2 - X_4,\;\; U_2 = X_3 - X_7,\;\; U_3 = X_6 - X_8.
\end{equation*}
The vector fields solving $X(K) = K$ and $X(K) = H$ respectively are
generated by
\begin{align*}
   V_1 &= \halfje (X_1 + X_5 + X_9)\\
   V_2 &= \halfje ((a_1 + b_1 I_2) X_1  + (a_2 + b_2 I_3) X_5
   + (a_3 + b_3 I_4) X_9).
\end{align*}
\end{lemma}

\begin{proof}
Let $X = \halfje \sum_{i=1}^9 \xi_i X_i$ then
\begin{equation*}
   X(K) = \xi_1 I_2^2 + \xi_5 I_3^2 + \xi_9 I_4^2
   + (\xi_2 + \xi_4) I_2 I_3 + (\xi_3 + \xi_7) I_2 I_4
   + (\xi_6 + \xi_8) I_3 I_4
\end{equation*}
and after some straightforward calculations the results follow.
\end{proof}

\subsection{The structure of the restricted tangent space}
\label{sec:tangentstruc}
The restricted tangent space~$T_1$, see section \ref{sec:rtan}, is the
sum of a module $\cM$ and an ideal $\cJ$ both subsets of~$\Rtdv$.
$\cM$ is a module over $\Rhk$ and generated by the functions $1$,
$G_1$ and~$G_2$.
$\cJ$ is the ideal generated by $F_1,F_2,F_3$.
So if $f \in T_1$ then
$f = \xi_1 F_1 + \xi_2 F_2 + \xi_3 F_3 + \eta_0 + \eta_1 G_1 + \eta_2 G_2$,
with $\xi_i \in \Rtdv$ and $\eta_i \in \Rhk$.

In lemma \ref{lem:vfs} we defined the vector fields
$U_1$, $U_2$, $U_3$, $V_1$ and $V_2$.
Thus we know the generators of $\cJ$ and~$\cM$
\begin{equation*}
\begin{aligned}
   F_1 &\becomes U_1(H) = (a_1 - a_2) I_2 I_3 + \hot\\
   F_2 &\becomes U_2(H) = (a_1 - a_3) I_2 I_4 + \hot\\
   F_3 &\becomes U_3(H) = (a_2 - a_3) I_3 I_4 + \hot\\
   G_1 &\becomes V_1(H) - H = H_6(I_2, I_3, I_4)\\
   G_2 &\becomes V_2(H) = a_1^2 I_2^2 + a_2^2 I_3^2 + a_3^2 I_4^2 + \hot
\end{aligned}
\end{equation*}
Defining $G_1$ as $V_1(H) - H$ instead of~$V_1(H)$ is just convenient
but not essential.
In the definition above we only show the \emph{leading terms} of
$F_1, \ldots, G_2$.

In principle each term in $f \in \Rtdv$ is an infinite series, but
with a term of lowest degree.
For our purposes it makes sense to call this the \emph{degree} of~$f$
and the term with lowest degree the \emph{leading term}.
Recall that the degree is at least~$1$ as we only consider formal
series without constant term.
Before using this to define a filtration on~$T_1$ we formally define
the degree of~$f$ and the leading term.
\begin{definition}
For $0 \neq f \in \Rtdv$ we define the \emph{degree} of $f$ as
$k \in \N$ for which $0 < \lim_{t \to 0} t^{-k} f(tI) < \infty$.
Suppose $k = \degree(f)$ then we call
$\sL(f) = \lim_{t \to 0} t^{-k} f(tI)$ the \emph{leading term}
of~$f$.
\end{definition}
The following properties of degree and leading term are almost
obvious.
\begin{lemma}
Let $f$ and $g$ be functions (germs) in~$\Rtdv$ and let $m$ and~$n$
be monomials in $\Rtdv$, then
\begin{enumerate}[i)]\itemsep 0pt
\item if $m(I) = I^k$ then $\degree(m) = |k|$
\item if $\degree(m) < \degree(n) = l$ then
$\degree(m+n) = \degree(m)$ and $\sL(m + n) = m$
\item $\degree(f+g) = \min(\degree(f), \degree(g))$ and if
$\degree(f) < \degree(g)$ then $\sL(f + g) = \sL(f)$
\item $\degree(f \cdot g) = \degree(f) \cdot \degree(g)$ and
$\sL(f \cdot g) = \sL(f) \sL(g)$
\end{enumerate}
\end{lemma}
With this notion of degree we define a filtration on~$\Rtdv$.
Since $\cJ$ and $\cM$ are subsets of $\Rtdv$ they immediately
inherit the filtration.
\begin{definition}\label{def:filt}
For $k \in \N_{>0}$ let $\cR_k$ be the set
$\{f \in \Rtdv \;|\; \degree(f) = k \}$.
Then we have $\cR_{k+1} \subset \cR_k$ and $\cR_1 = \Rtdv$,
therefore $\cR_k$ is a filtration of $\Rtdv$.
Similarly $\{\cJ_k\}$ and $\{\cM_k\}$ are filtrations.
\end{definition}

\begin{remark}\label{rem:groebner}
As an analogy of a Gr\"obner basis for polynomial ideals,
see~\cite{bhlv2003}, we could hope that $T_1$ is generated by
$\sL(F_1), \sL(F_2), \sL(F_3), 1, \sL(G_1), \sL(G_2)$ in the
following sense: every $f \in T_1$ can be written as
$\xi_1 \sL(F_1) + \xi_2 \sL(F_2) + \xi_3 \sL(F_3)
 + \eta_0 + \eta_1 \sL(G_1) + \eta_2 \sL(G_2)$,
with $\xi_i \in \Rtdv$ and $\eta_i \in \Rhk$.
\end{remark}

\subsection{Splitting into homogeneous parts}\label{sec:homsplit}
Since the co-dimension of~$\cF$ as a smooth mapping is the same as
the co-dimension of the mapping as a formal power series, we can
simplify the problem by looking at homogeneous functions and add
the co-dimensions found for each degree starting at degree one.
This is carried out in the following chain of assertions.

Let $\cHtdv{k}$ be the set of all homogeneous functions of
degree~$k$ in $I_2$, $I_3$ and $I_4$.
In fact we have $\cHtdv{k} = \cR_k/\cR_{k+1}$.
Furthermore let $\cHhk{k}$ be the set of all homogeneous functions
of degree $k$ in $K$ and $H$, then $\cHhk{m} \subset \cR_{2m}$.
Since $\cHhk{m}$ is not homogeneous in $I$ we use a projection
$\Pi_k : \cR_k \to \cHtdv{k}$ selecting the homogeneous part of a
function $f \in \cR_k$.
The following general result leaves us with a small number of cases.
\begin{proposition}\label{pro:degk}
The co-dimension of $\Pi_k(T_1)$ in $\cHtdv{k}$ is zero for $k = 4$
and $k \geq 6$.
Or, put differently, the mapping (odd degree)
\begin{gather*}
   \cHtdv{2m-1}^3 \times \cHhk{m-1} \to \cHtdv{2m+1}:\\
   (\xi_1, \xi_2, \xi_3, \eta_1) \mapsto \Pi_{2m+1} \big(
   \xi_1\sL(F_1) + \xi_2\sL(F_2) + \xi_3\sL(F_3) + \eta_1\sL(G_1)
   \big)
\end{gather*}
is onto for $m \geq 3$ and also the mapping (even degree)
\begin{gather*}
   \cHtdv{2m-2}^3 \times \cHhk{m} \times \cHhk{m-1} \to \cHtdv{2m}:\\
   (\xi_1, \xi_2, \xi_3, \eta_0, \eta_2) \mapsto \Pi_{2m} \big(
   \xi_1\sL(F_1) + \xi_2\sL(F_2) + \xi_3\sL(F_3) + \eta_0 + \eta_2\sL(G_2)
   \big)
\end{gather*}
is onto for $m \geq 2$.
\end{proposition}
Thus we have to investigate degrees $1, 2, 3$ and~$5$ separately.
First we prove proposition~\ref{pro:degk} in three lemmas.
In order to do so it is useful to introduce some notation, which
is motivated by the fact that the projection of
$\xi_1\sL(F_1) + \xi_2\sL(F_2) + \xi_3\sL(F_3)$ on
$\spansel{I_2^k, I_3^k, I_4^k}$ is always zero.
\begin{definition}\label{def:bbb}
The space $\cHtdv{k}$ has a monomial basis denoted by
$b_k = \basis{\ldots, I_2^k, I_3^k, I_4^k}$.
Let $b^{\sharp}$ be the set of monomials $I_2^k$, $I_3^k$ and~$I_4^k$.
Furthermore let $b_k^{\flat}$ be the set of monomials in~$b_k$ with
the monomials in $b^{\sharp}$ \emph{excluded}.
Finally let $B_k^{\sharp}$ be the subspace of $\cHtdv{k}$ spanned
by~$b_k^{\sharp}$, similarly $B_k^{\flat}$ is spanned by~$b_k^{\flat}$.
\end{definition}
The next three lemmas treat different parts of
proposition~\ref{pro:degk}.
The following lemma shows that the mapping from $\cHtdv{k-2}^3$
to~$B_k^{\flat}$ is onto for each $k \geq 2$.
Thus we get rid of the first factor of the mapping in
proposition~\ref{pro:degk}.
Later on we use this lemma again for the remaining low degree cases.
\begin{lemma}\label{lem:bplat}
The mapping
$\cHtdv{k-2}^3 \to B_k^{\flat} : (\xi_1, \xi_2, \xi_3) \mapsto
 \xi_1 \sL(F_1) + \xi_2 \sL(F_2) + \xi_3 \sL(F_3)$
is onto provided that $a_1 - a_2 \neq 0$, $a_2 - a_3 \neq 0$ and
$a_3 - a_1 \neq 0$ and $k \geq 2$.
\end{lemma}
\begin{proof}
Every monomial in $B_k^{\flat}$ can be written as either
$I^l I_2 I_3$, $I^l I_3 I_4$ or $I^l I_2 I_4$ for some
multi-index~$l$ with $|l| = k-2$.
Therefore every $f \in B_k^{\flat}$ can be expressed as
$\xi_1 \sL(F_1) + \xi_2 \sL(F_2) + \xi_3 \sL(F_3)$ for some
$\xi_i \in \cHtdv{k-2}$, but only if $a_1 - a_2 \neq 0$,
$a_2 - a_3 \neq 0$ and $a_3 - a_1 \neq 0$.
If for example $a_1 - a_2 = 0$, then $I_2 I_3 \not \in T_1$.
\end{proof}
The following two lemmas show that the second factor of the mapping
in proposition \ref{pro:degk} maps onto $B^{\sharp}$, but we have to
distinguish the odd and even degree cases.
\begin{lemma}[Odd degree]\label{lem:xyzodd}
The mapping
$\cHhk{m-1} \to \cHtdv{2m+1} : \eta_1 \mapsto
 \Pi_{2m+1} \big( \eta_1 \sL(G_1) \big)$
followed by projection on $B_{2m+1}^{\sharp}$ is onto provided that
 $a_1 - a_2 \neq 0$, $a_2 - a_3 \neq 0$, $a_3 - a_1 \neq 0$ and
$b_1 b_2 b_3 \neq 0$ and $m \geq 3$.
\end{lemma}
\begin{proof}
The projection of the functions
$K^{m-1}\sL(G_1), K^{m-2}H\sL(G_1), \ldots, H^{m-1}\sL(G_1)$
on $B_{2m+1}^{\sharp}$ is given by the vectors in the matrix
\begin{equation*}
\begin{pmatrix}
   b_1 & a_1 b_1 & a_1^2 b_1 & & a_1^{m-1} b_1 \\
   b_2 & a_2 b_2 & a_2^2 b_2 & \ldots & a_2^{m-1} b_2 \\
   b_3 & a_3 b_3 & a_3^2 b_3 & & a_3^{m-1} b_3
\end{pmatrix}
\end{equation*}
which has rank three as soon as the conditions are met.
\end{proof}
Finally we state and prove a lemma for the even degree case.
\begin{lemma}[Even degree]\label{lem:xyzeven}
The mapping
$\cHhk{m} \times \cHhk{m-1} \to \cHtdv{2m} : (\eta_0, \eta_2)
 \mapsto \Pi_{2m}\big( \eta_0 + \eta_2 \sL(G_2) \big)$
followed by projection on $B_{2m}^{\sharp}$ is onto provided that
$a_1 - a_2 \neq 0$, $a_2 - a_3 \neq 0$ and $a_3 - a_1 \neq 0$ and
$m \geq 2$.
\end{lemma}
\begin{proof}
The projection of the functions
\begin{equation*}
   K^m, K^{m-1} H, \ldots, H^m, K^{m-1} \sL(G_2), K^{m-2} H \sL(G_2),
   \ldots, H^{m-1} \sL(G_2)
\end{equation*}
on~$B_{2m}^{\sharp}$ is given by the vectors in the matrix
\begin{equation}
\label{matrix:projection}
\begin{pmatrix}
   1 & a_1 & a_1^2 & & a_1^m \\
   1 & a_2 & a_2^2 & \ldots & a_2^m \\
   1 & a_3 & a_3^2 & & a_3^m
\end{pmatrix}
\end{equation}
which has rank three as soon as the conditions are met.
\end{proof}
With these three lemmas we prove proposition \ref{pro:degk}.
\begin{proof}[Proof of proposition \ref{pro:degk}]
The odd degree part of the proposition is covered by combining lemmas
\ref{lem:bplat} and~\ref{lem:xyzodd} showing that the product mapping
is onto~$\cHtdv{2m+1}$.
Similarly combining lemmas \ref{lem:bplat} and~\ref{lem:xyzeven} shows
that in case of even degree the product mapping is onto~$\cHtdv{2m}$.
\end{proof}
Finally we consider the remaining cases: degrees $1, 2, 3$ and~$5$.
In all cases we follow the same pattern, we determine the
co-dimension of
$\Pi_k \big( \xi_1 F_1 + \xi_2 F_2 + \xi_3 F_3
 + \eta_0 + \eta_1 G_1 + \eta_2 G_2 \big)$
in $\cHtdv{k}$ for $k \in \{1,2,3,5\}$.
But in view of lemma~\ref{lem:bplat} we only have to consider the
projection on~$B_k^{\sharp}$.
The main result of this part is the next proposition.
\begin{proposition}\label{pro:degonetofive}
A complement of $T_1$ in $\cR$ is spanned by the functions
$\spansel{I_2, I_3, I_4, I_2^3, I_3^3}$ or
$\spansel{I_2, I_3, I_4, I_2^3, I_4^3}$ or
$\spansel{I_2, I_3, I_4, I_3^3, I_4^3}$ as a linear space.
\end{proposition}
We prove this proposition in several lemmas.
The following lemma is immediately clear.
\begin{lemma}[Degree one]\label{lem:degone}
A monomial basis of functions of degree one is
$\basis{I_2, I_3, I_4}$.
Since $T_1$ does not contain functions of degree one, the
co-dimension in this space is three and a complement is
$B_1^{\sharp}$ itself.
\end{lemma}
Thus we get unfolding terms: $\mu_1 I_2$, $\mu_2 I_3$
and~$\mu_3 I_4$.
\begin{lemma}[Degree two]\label{lem:degtwo}
Functions of degree two with a nonzero projection on $B_2^{\sharp}$
are $K$, $H$ and~$G_2$.
These three functions are independent a soon as
$(a_1 - a_2)(a_2 - a_3)(a_3 - a_1) \neq 0$.
\end{lemma}
\begin{proof}
The projection of $K$, $H$ and $G_2$ onto $B_2^{\sharp}$ is given by
the matrix
\begin{equation*}
   A_2^{\sharp} =
   \begin{pmatrix}
      1 & a_1 & a_1^2 \\
      1 & a_2 & a_2^2 \\
      1 & a_3 & a_3^2
   \end{pmatrix}
   \enspace ,
\end{equation*}
cf.~\eqref{matrix:projection}.
The determinant of $A_2^{\sharp}$ is
$(a_1 - a_2)(a_2 - a_3)(a_3 - a_1)$.
\end{proof}
\begin{lemma}[Degree three]\label{lem:degthree}
There is only one function in $T_1$ with a nonzero projection
on~$B_3^{\sharp}$, nam\-e\-ly~$G_1$.
Thus the co-dimension of $T_1$ in the space of homogeneous functions
of degree three is two.
As a complement any pair of $I_2^3$, $I_3^3$ and $I_4^3$ will do.
We take for example $\mu_4 I_2^3$ and $\mu_5 I_3^3$ as unfolding
terms, then we must impose the condition $b_3 \neq 0$.
\end{lemma}
\begin{proof}
The projection of $G_1$, $\mu_4 I_2^3$ and $\mu_5 I_3^3$
on~$B_3^{\sharp}$ is given by the matrix
\begin{equation*}
   A_3^{\sharp} =
   \begin{pmatrix}
      b_1 & \mu_4 & 0 \\
      b_2 & 0 & \mu_5 \\
      b_3 & 0 & 0
   \end{pmatrix}
\end{equation*}
\end{proof}
\begin{lemma}[Degree five]\label{lem:degfive}
There are only two functions of degree five in~$T_1$, namely $KG_1$
and~$HG_1$, with a nonzero projection on~$B_5^{\sharp}$.
However, a function $F_5 \in T_1$ exists such that
$\Pi_k\big(F_5\big) = 0$ for $k \leq 4$ and $\Pi_5\big(F_5\big) \neq 0$.
With $F_5$ the co-dimension of~$T_1$ in the space of homogeneous
functions of degree five is zero, provided that
$(a_1 - a_2)(a_2 - a_3)(a_3 - a_1) b_1 b_2 b_3 \neq 0$.
\end{lemma}
\begin{proof}
Let
\begin{equation*}
   F_5 = \xi_1 I_2 I_3 F_1 + \xi_2 I_2 I_4 F_2
   + \xi_3 I_3 I_4 F_3 + \eta_{01} K^2 + \eta_{02} KH
   + \eta_{03} H^2 + \eta_{21} KG_2 + \eta_{22} HG_2
\end{equation*}
be a function of degree 4, with $\xi_1,\ldots,\eta_{22} \in \R$.
Then a non-trivial solution of $\Pi_4 \big( F_5 \big) = 0$ exists
while $\Pi_5 \big( F_5 \big) \neq 0$.
The projection of the functions $\Pi_5 \big( KG_1 \big)$,
$\Pi_5 \big( HG_1 \big)$ and $\Pi_5 \big( F_5 \big)$
onto~$B_5^{\sharp}$ has the matrix
\begin{equation*}
   A_5^{\sharp} =
   \begin{pmatrix}
      b_1 & a_1 b_1 & -(a_1^2 - a_2 a_3 + a_1 (a_2 + a_3)) b_1\\
      b_2 & a_2 b_2 & -(a_2^2 - a_1 a_3 + a_2 (a_1 + a_3)) b_2\\
      b_3 & a_3 b_3 & -(a_3^2 - a_1 a_2 + a_3 (a_1 + a_2)) b_3\\
   \end{pmatrix}
\end{equation*}
and
$\det(A_5^{\sharp}) = (a_1 - a_2)(a_2 - a_3)(a_3 - a_1) b_1 b_2 b_3$.
\end{proof}
The last lemma is about the modal parameters.
\begin{lemma}\label{lem:modal}
Parameters $\mu_4$ and $\mu_5$ are moduli.
\end{lemma}
\begin{proof}
Let $H(I;\mu)$ be as in the main theorem~\ref{the:uniunfo}.
From the previous proofs it follows almost immediately that the
unfoldings of $(H(I;0,0,0,0,0), K)$ and $(H(I;0,0,0,\mu_4,\mu_5), K)$
are equal for small values of $\mu_4$ and~$\mu_5$.
Therefore $\mu_4$ and $\mu_5$ are moduli.
\end{proof}
The proof of theorem \ref{the:uniunfo} follows from proposition
\ref{pro:degk}, proposition \ref{pro:degonetofive} and lemma
\ref{lem:modal}.
\begin{remark}
As a by product we find that $T_1$ is not generated by
\begin{equation*}
   \sL(F_1), \sL(F_2), \sL(F_3), 1, \sL(G_1), \sL(G_2).
\end{equation*}
See remark \ref{rem:groebner}.
\end{remark}
%

\section{Discussion}\label{sec:disco}
The dynamics of an $n$--degree-of-freedom Hamiltonian system locally
around an elliptic equilibrium at the origin is characterised by an
$n$--tuple $\omega \in \R^n$ of frequencies.
When the frequencies satisfy an integer relation
$\inprod{m}{\omega} \neq 0$ with $m \in \Z^n$ we say that the
frequencies are \emph{resonant}.
For most equilibria the frequencies are non-resonant.
However, when the system depends on parameters there are resonances
at a dense subset of parameter values.
Since low order resonances are accompanied by bifurcations the
corresponding points in parameter space are of special interest.

Here we consider two-degree-of-freedom systems.
In that case $\omega = (\omega_1, \omega_2)$, so $\omega$ is resonant
if $\omega_1 / \omega_2$ is an element of~$\Q$.
We may assume without loss of generality that $\omega_1$ and~$\omega_2$
are relative prime integers $k$ and~$l$ at resonance.
The linear part of the vector field is determined by $\omega = (k, l)$
if $k \neq \pm l$.
In linear Hamiltonian systems imaginary eigenvalues, in casu the
frequencies $k, l$ have a sign.
The sign is related to the Morse index of the Hamiltonian.
Therefore a $k{:}l$~resonance is not equivalent to a $k{:}{-}l$~resonance;
in particular the $1{:}1$ and $1{:}{-}1$~resonances are not equivalent.
Moreover, eigenvalues with equal sign are always semi-simple, whereas
the $1{:}{-}1$~resonance can also be nilpotent.
Thus there are three resonances with equal frequencies, namely the
semi-simple $1{:}{-}1$, the nilpotent $1{:}{-}1$ and the $1{:}1$~resonance.
The latter is always semi-simple.
The nilpotent $1{:}{-}1$~resonance is what triggers the Hamiltonian Hopf
bifurcation.

As indicated in the introduction the $k{:}l$~resonances, with
$k, l \in \N$, are very similar.
In particular, in the sense of section~\ref{sec:mapstab} the
co-dimension is~$2$, provided that $k{:}l$ is not equal to $1{:}1$,
$1{:}2$ or~$1{:}3$.
The last two exceptional cases have co-dimension $1$ and~$3$, respectively.
Thus all definite resonances except $1{:}1$ have in common that they
occur persistently in $1$--parameter families and if more parameters
are present these are moduli, see~\cite{duistermaat1984}.
In this respect our case of the $1{:}1$~resonance is very exceptional:
its co-dimension is~$5$, it occurs persistently in $3$--parameter
families and two of the unfolding parameters are moduli.
When we restrict to the linear unfolding, there is a transformation
group acting on the unfolding.
This can be used to reduce the number of parameters.
Using invariants of this transformation group we find that one of the
generators is $\mu_1^2 + \mu_2^2 + \mu_3^2$.
Then in a \emph{reduced  linear unfolding} the $1{:}1$~resonance
occurs persistently in a $1$--parameter family, see~\cite{hk2010}
for more details.

Before applying singularity theory we reduce the $S^1$--symmetric
system using invariants.
Another approach is that in~\cite{bhlv2003} where the system is
first reduced to a planar system.
Then singularity theory using right equivalence is applied to obtain
an unfolding.
With a different notion of equivalence one may expect different
co-dimensions.
In~\cite{bhlv2003}, by nature of the method, one finds lower bounds
for the co-dimensions.
For the resonances $1{:}2$, $1{:}3$ and $1{:}4$ these lower bounds
are computed and they coincide with the co-dimensions found
in~\cite{duistermaat1984}, namely $1$, $3$ and~$2$, respectively.
However, the non-degeneracy conditions of \cite{bhlv2003}
and~\cite{duistermaat1984} differ.
It would be interesting to compare both methods for the
$1{:}1$~resonance.

The results obtained so far are a starting point for extensions and
applications.
Let us list a few.
In general, when a system passes a resonance upon varying one or more
parameters, one expects a bifurcation to occur.
We see this phenomenon in the resonances mentioned earlier.
Therefore we would like to explore the bifurcation scenario of the
$1{:}1$~resonance, or more general explore the geometry of level sets
of the momentum mapping depending on parameters near $1{:}1$~resonance.
A similar program can be carried out for Hamiltonian systems in
$1{:}1$~resonance which are also reversible, see~\cite{lisse1995},
or symmetric (other than the $S^1$ symmetry induced by the
$1{:}1$~resonance).
The unfolding of the semisimple $1{:}{-}1$~resonance is similar to the
unfolding of the $1{:}{-}1$~resonance, but the bifurcation scenario is
most likely very different.
A well-known system in $1{:}1$~resonance is the H\'enon--Heiles system.
Our original plan, to apply the unfolding and bifurcation results, now
comes within reach.
Furthermore we wish to relate our results to the results in a series of
articles by Elipe, Lanchares et al.\ and Frauendiener
\cite{elis1995, frauendiener1994,
 lanchares1993, le1994, le1995a, le1995b, lisse1995}
for families of $S^1$--symmetric Hamiltonian systems.
These are the subjects of future publications.

\section*{Acknowledgment}
It is a pleasure to thank Henk Broer, Richard Cushman, Jaap Top and
Gert Vegter for fruitful discussions and suggestions.

%
\bibliographystyle{plain}

\begin{thebibliography}{10}

\bibitem{am1987}
R.H. Abraham and J.E. Marsden.
\newblock {\em Foundations of mechanics}.
\newblock Addison-Wesley, $2$nd edition, 1987.

\bibitem{arnold1980}
V.I. Arnol'd.
\newblock {\em Mathematical methods of classical mechanics}.
\newblock Springer-Verlag, New York, 1980.

\bibitem{bl1975}
Th. Br\"ocker and L.~Lander.
\newblock {\em Differentiable Germs and Catastrophes}.
\newblock Cambridge Univ. Press, 1975.

\bibitem{bhlv2003}
H.~W. Broer, I.~Hoveijn, G.~Lunter, and G.~Vegter.
\newblock {\em Bifurcations in {Hamiltonian} systems, Computing Singularities
  by {Gr\"obner} Bases}, volume LNM 1806 of {\em Lecture Notes in Mathematics
  series}.
\newblock Springer-Verlag, New York, 2003.

\bibitem{bc1977}
N.~Burgoyne and R.~C. Cushman.
\newblock Normal forms for real linear hamiltonian systems.
\newblock In C.~Martin and R.~Hermann, editors, {\em The 1976 Ames Research
  Center (NASA) Conference on Geometric Control Theory}, pages 483--529. Math.
  Sci. Press., Brookline, Mass., 1977.

\bibitem{cotter1986}
C.~Cotter.
\newblock {\em The $1{:}1$ semisimple resonance}.
\newblock PhD thesis, University of California at Santa Cruz, 1986.

\bibitem{cb2015}
R.H. Cushman and L.~Bates.
\newblock {\em Global Aspects of Classical Integrable Systems}.
\newblock Birkhauser, $2$nd edition, 2015.

\bibitem{cdhs2007}
R.C. Cushman, H.R. Dullin, H.~Han{\ss}mann, and S.~Schmidt.
\newblock The $1{:}{\pm} 2$ resonance.
\newblock In J.J. Duistermaat and H.~Han{\ss}mann, editors, {\em Dynamics and
  Hamiltonian systems, Utrecht 2007}, volume~12, pages 640--661. R \& C
  Dynamics, 2007.

\bibitem{cr1998}
R.H. Cushman and D.L. Rod.
\newblock Reduction of the semi-simple $1{:}1$ resonance.
\newblock {\em Physica D}, 6:105--112, 1998.

\bibitem{duistermaat1984}
J.J. Duistermaat.
\newblock Bifurcations of periodic solutions near equilibrium points of
  hamiltonian systems.
\newblock In L.~Salvadori, editor, {\em Bifurcation theory and applications,
  Montecatini 1983}, Lecture Notes in Mathematics 1057. Springer-Verlag, New
  York, 1984.

\bibitem{elis1995}
A.~Elipe, V.~Lanchares, M~I\~narrea, and J.P. Salas.
\newblock Triparametric bifurcations in a quadratic hamiltonian.
\newblock In M.C. L\'opez~de Silanes, M.~San~Miguel, G.~Sanz, and
  M.~Madaune-Tor, editors, {\em Act. IV. Journ\'ees Zaragoza-Pau Math. Appl.,
  Jaca 1995}, pages 149--158. Publ. Univ. Pau, 1997.

\bibitem{frauendiener1994}
J.~Frauendiener.
\newblock Quadratic hamiltonians on the unit sphere.
\newblock {\em Mech. Res. Comm.}, 22:313--317, 1994.

\bibitem{hanssmann2007}
H.~Han{\ss}mann.
\newblock {\em Local and Semi--Local Bifurcations in Hamiltonian Dynamical
  Systems --- Results and Examples}, volume 1893 of {\em LNM}.
\newblock Springer, 2007.

\bibitem{hhxxxx}
H.~Han{\ss}mann and I~Hoveijn.
\newblock The semi-simple $1$:$-1$ resonance in {H}amiltonian systems.
\newblock in preparation.

\bibitem{hm2003}
H.~Han{\ss}mann and J.C. van~der Meer.
\newblock On non-degenerate hamiltonian hopf bifurcations in 3dof~systems.
\newblock In F.~(Dumortier, editor, {\em Equadiff 2003, Hasselt}, pages
  476--481. World Scientific, 2005.

\bibitem{hoveijn1996a}
I.~Hoveijn.
\newblock Versal deformations and normal forms for reversible and {Hamiltonian}
  linear systems.
\newblock {\em Journal of Differential Equations}, 126(2):408--442, 1996.

\bibitem{hk2010}
I.~Hoveijn and O.~N. Kirillov.
\newblock Singularities on the boundary of the stability domain near
  1:1-resonance.
\newblock {\em J. Differential Equations}, 248(10):2585--2607, 2010.

\bibitem{hlr2003}
I.~Hoveijn, J.~S.~W. Lamb, and R.~M. Roberts.
\newblock Normal forms and unfoldings of linear systems in eigenspaces of
  (anti)-automorphisms of order two.
\newblock {\em Journal of Differential Equations}, 190:182--213, 2003.

\bibitem{kummer1976}
M.~Kummer.
\newblock On resonant nonlinearly coupled oscillators with two equal
  frequencies.
\newblock {\em Comm. Math. Phys.}, 48:53--79, 1976.

\bibitem{kummer1978}
M.~Kummer.
\newblock On resonant classical hamiltonians with two equal frequencies.
\newblock {\em Comm. Math. Phys.}, 58:85--112, 1978.

\bibitem{lanchares1993}
V.~Lanchares.
\newblock {\em Sistemas din\'amicos bajo la acci\'on del grupo $SO(3)$: El caso
  de un Hamiltoniano cuadr\'atico}.
\newblock PhD thesis, Universidad de Zaragoza, 1993.

\bibitem{le1994}
V.~Lanchares and A.~Elipe.
\newblock Biparametric quadratic hamiltonians on the unit sphere: complete
  classification.
\newblock {\em Mech. Res. Comm.}, 21:209--214, 1994.

\bibitem{le1995a}
V.~Lanchares and A.~Elipe.
\newblock Bifurcations in biparametric quadratic potentials.
\newblock {\em Chaos}, 5:367--373, 1995.

\bibitem{le1995b}
V.~Lanchares and A.~Elipe.
\newblock Bifurcations in biparametric quadratic potentials ii.
\newblock {\em Chaos}, 5:531--535, 1995.

\bibitem{lisse1995}
V.~Lanchares, M.~I\~narrea, J.P. Salas, J.D. Sierra, and A.~Elipe.
\newblock Surfaces of bifurcation in a triparametric quadratic hamiltonian.
\newblock {\em Phys. Rev. E}, 52:5540--5548, 1995.

\bibitem{martinet1982}
M.~Martinet.
\newblock {\em Singularities of smooth functions and maps}.
\newblock London Math. Soc. Lecture Note Series 58. Cambridge University Press,
  1982.

\bibitem{meer1985}
J.~C. van~der Meer.
\newblock {\em The Hamiltonian Hopf bifurcation}.
\newblock Number 1160 in Lecture Notes in Mathematics. Springer-Verlag, New
  York, 1985.

\bibitem{montaldi2013}
J.~Montaldi.
\newblock Singularities, bifurcations \& catastrophes.
\newblock preprint, 2013.

\bibitem{poenaru1976}
V.~Po\'enaru.
\newblock {\em Singularit\'es $C^{\infty}$ en Pr\'esence de Sym\'etrie}.
\newblock Lecture Notes in Mathematics 510. Springer-Verlag, New York, 1976.

\bibitem{sanders1978}
J.~Sanders.
\newblock Are higher order resonances really interesting?
\newblock {\em Celestial Mechanics}, 16:421--440, 1978.

\end{thebibliography}

\end{document}